\documentclass[leqno,twoside, 11pt]{amsart}
\usepackage{amssymb, latexsym, esint}
\usepackage{color}

\usepackage{amsthm}
\usepackage{amssymb,amsfonts}
\usepackage[all,arc]{xy}
\usepackage{enumerate}
\usepackage{mathrsfs}
\usepackage{amsmath}
\usepackage{verbatim}

\theoremstyle{plain}
\numberwithin{equation}{section} \numberwithin{figure}{section}

\newtheorem{theorem}{Theorem}[section]
\newtheorem{lemma}[theorem]{Lemma}
\newtheorem{proposition}[theorem]{Proposition}
\newtheorem{corollary}[theorem]{Corollary}
\newtheorem{definition}[theorem]{Definition}
  \usepackage{epsfig} 
\usepackage{graphicx}  \usepackage{epstopdf}
\theoremstyle{definition}
\newtheorem{remark}[theorem]{Remark}

\numberwithin{equation}{section}

\begin{document}

\markboth{Pablo Ochoa and Analia Silva}
{Magnetic fractional p-Laplacian}

%

%

\title[Biharmonic $g$-Laplacian]{Nonlinear eigenvalue problems for a biharmonic operator in Orlicz-Sobolev spaces}

\author{Pablo Ochoa and Anal\'ia Silva}

\address[P.Ochoa]{Universidad Nacional de Cuyo-CONICET, Parque Gral. San Mart\'in\\
Mendoza, 5500, Argentina.\\
ochopablo@gmail.com}

\address[A.Silva]{Departamento de Matem\'atica, FCFMyN, Universidad Nacional de San Luis and Instituto de Matem\'atica Aplicada San luis (IMASL),
UNSL-CONICET. Ejercito de los Andes
950, D5700HHW,San Luis,
Argentina\\acsilva@unsl.edu.ar\\analiasilva.weebly.com}


\maketitle
\begin{abstract}
In this paper, we introduce a new higher-order Laplacian operator in the framework of Orlicz-Sobolev spaces, the biharmonic g-Laplacian
$$\Delta_g^2 u:=\Delta \left(\dfrac{g(|\Delta u|)}{|\Delta u|} \Delta u\right),$$
where $g=G'$, with $G$ an N-function. This operator is a generalization of the so called bi-harmonic Laplacian $\Delta^2$. Here, we also established  basic functional properties of $\Delta_g^2$, which can be applied to  existence results. Afterwards, we study the eigenvalues of $\Delta_g^2$, which depend on normalisation conditions, due to the lack of homogeneity of the operator. Finally, we study different nonlinear eigenvalue problems associated to $\Delta_g^2$ and we show  regimes where the corresponding spectrum concentrate at $0$, $\infty$ or coincide with $(0, \infty)$. 
\end{abstract}

\section{Introduction}
In this paper, we introduce a new higher-order operator in the framework of Orlicz-Sobolev spaces that generalizes the well-known biharmonic Laplacian:
$$\Delta^2u=\Delta(\Delta u).$$Given $\Omega$ a bounded domain of $\mathbb{R}^n$ and  a function $u:\Omega\to \mathbb{R}$, we formally define  the biharmonic $g$-Laplacian as
$$\Delta_g^2 u= \Delta \left(\dfrac{g(|\Delta u|)}{|\Delta u|} \Delta u\right),$$
where $g=G'$, with $G$ an N-function, that is, $G:\mathbb{R}\to [0,\infty)$ is even and is given by
$$G(t)=\int_0^t g(s)\,ds,$$where $g$ is odd,  continuous, non-decreasing, $g(0)=0$, $g(t)>0$ for $t>0$, and $g(t)\to \infty$  as $t\to \infty$.

In this article, we will discuss basic properties of $\Delta_g^2$ related to the solvability and uniqueness of elliptic problems driven by this operator. Among them, we shall prove  existence and continuity of the derivative of the energy functional $L$ associated to $\Delta_g^2$. Monotonicity properties, such as the $(S_+)$-condition, together with the continuity of the inverse of the derivative $L'$ will also be established. As it is well-known in the literature, this list of properties may be applied to study several problems related to $\Delta_g^2$.

 Moreover, we shall also consider the existence of eigenvalues and eigenfunctions and their relation to variational problems. As it is known in this non-standard framework (see for instance \cite{GH}, \cite{MT}, \cite{YM}), eigenvalues are not variational, that is, they may not be characterized by taking infimun of Rayleigh quotients. This is due in general to the lack of homogeneity of $N$-functions and, as a particular feature of eigenvalue problems in Orlicz spaces, the solvability of the eigenvalue problems depends on normalization conditions. To illustrate this observation,  in the case of the $g$-Laplacian operator, it has been proved  in \cite{MT} that the constrained minimization problem
 $$\inf\left\lbrace \int_\Omega G(|\nabla u|)\,dx, \int_\Omega G(u)\,dx=r \right\rbrace$$has a solution $u_r\neq 0$ which is a weak solution of
 \begin{equation}\label{problem eigenvalue intro2}
\begin{cases}-\Delta_g u = \lambda_r \dfrac{g(|u|)}{|u|}u,\quad  \text{in }\Omega\\
u=0\qquad \normalcolor\text{ on }\partial \Omega,
\end{cases} 
\end{equation}where $g=G'$, for some $\lambda_r\in \mathbb{R}$. We point out that the eigenvalues depend on the normalization condition
$$\int_\Omega G(u)\,dx=r.$$
Observe that when $G(t)=|t|^p$, $p>1$, problem \eqref{problem eigenvalue intro2} becomes the more familiar eigenvalue problem for the $p$-Laplacian:
\begin{equation*}
\begin{cases}-\Delta_p u = \lambda |u|^{p-2}u,\quad  \text{in }\Omega\\
u=0\qquad \normalcolor\text{ on }\partial \Omega,
\end{cases} 
\end{equation*}where $\Delta_p u:= div(|\nabla u|^{p-2}\nabla u)$. See for instance \cite{Lind} and \cite{Lind2}, and the references therein.

In  \cite{T}, the author  proved the existence of an infinite sequence of eigenvalues for \eqref{problem eigenvalue intro2} tending to infinity without any extra condition on $G$ (no $\Delta_2$-condition is assumed). There, the classical Ljusternik-Schnirelmann theory is not available due to nonreflexivity of the underlying Orlicz-Sobolev space.  Applications of Mathematical Programming to eigenvalue problems in Orlicz-Sobolev spaces can be seen in \cite{GM}.   More recently, the isolation of the principal eigenvalue $\lambda_1$ for the $g$-Laplacian has been considered in \cite{MLo} (see also \cite{YM} for further properties of the eigenfunctions).

   In this setting, our contributions are the following: given the nonlinear eigenvalue problem in Orlicz-Sobolev spaces
 \begin{equation}\label{problem eigenvalue intro}
\begin{cases}\Delta_g^2 u = \lambda \dfrac{b(|u|)}{|u|}u,\qquad  \text{in }\Omega\\
u\in W^{2, G}_0(\Omega),
\end{cases} 
\end{equation}where $g=G', b=B'$, with $B$ and $G$  two $N$-functions, we prove:
 
 \begin{itemize}
 \item[(i)] When $G=B$, and for each normalization constraint
 $$\int_\Omega G(u)\,dx=r, \,\,r>0,$$there is $\lambda_r>0$ and $u_r\neq 0$ solving \eqref{problem eigenvalue intro}. Moreover, any eigenvalue $\lambda_r$ of \eqref{problem eigenvalue intro} is necessary greater or equal than the following infimum of \textit{Rayleigh quotients}
 $$\lambda_0=\inf_{u\neq 0}\dfrac{\int_\Omega g(|\Delta u|)|\Delta u|\,dx}{\int_\Omega g(| u|)| u|\,dx}.$$In our framework, it is an open question  if $\lambda_0$ is indeed an eigenvalue of \eqref{problem eigenvalue intro}, even for domains with small diameters (compare to Theorem 3.3 in \cite{GH} where thanks to a Poincar\'e inequality the corresponding $\lambda_0$ is positive).
 \item[(ii)] The spectrum is the whole interval $(0, \infty)$ in two cases. The first one, is when $B$ grows much faster than $G$ near $\infty$, much slower than $G$ near $0$, but it is still subcritical (in the sense of the embedding between Orlicz-Sobolev and Orlicz spaces). This is done by showing that the energy funcional associated with \eqref{problem eigenvalue intro} has a non-trivial global minimum. See Theorem \ref{eigenvalue3}.
  The second case, is when $B$ grows much slowly than $G$  near $\infty$ and much faster near $0$. Here, the coerciveness of the energy functional is essential to find a non-trivial critical point. See Theorem \ref{eigenvalueinfty}.

 \item[(iii)] When $G$ grows much slowly than $B$ around $0$, the spectrum is concentrated around $0$, that is, there is $\lambda^{*}>0$ such that any $\lambda\in (0, \lambda^{*})$ is an eigenvalue of \eqref{problem eigenvalue intro}. This is done by showing that the associated energy functional has a mountain pass geometry, but since no Ambrosetti-Rabinowitz condition is satisfied, the conclusion is obtained by the Ekeland's variational principle and functional properties of the main operator $\Delta_g^2$. We refer to Theorem \ref{eigenvalue}.
 
 \item [(iv)] Finally, when $B$ grows much slowly than $G$ near $0$, then the spectrum is concentrated around $\infty$, that is, there is $\lambda^{**}>0$ such that any $\lambda>\lambda^{**}$ is an eigenvalue of \eqref{problem eigenvalue intro}.   We refer to Theorem \ref{eigenvalueinfty2}.
 \end{itemize}

%
  
  The paper is organized as follows. In Section \ref{prelim}, we give definitions and useful results in Orlicz and Orlicz-Sobolev spaces. In section \ref{weak} , we motivate and discuss the definition of weak solutions to  Dirichlet problem involving the biharmonic $g$-Laplacian. Next, in Section \ref{prop}, we present some functional properties  for  $\Delta_g^2$. In Section \ref{g}, we studied the eigenvalue problem \eqref{problem eigenvalue intro2} and, afterwards, in Section \ref{b}, we analize the behaviour of the spectrum of problem \eqref{problem eigenvalue intro} under different regimes.

\section{Preliminaries}\label{prelim}

In this section we introduce basic definitions and preliminary results related to Orlicz spaces. We start recalling the definition of an N-function.
	\begin{definition}\label{d2.1}
		A function $G \colon [0, \infty) \rightarrow \mathbb{R}$ is called an N-function if it admits the representation
		$$G(t)= \int _{0} ^{t} g(\tau) d\tau,$$
		where the function $g$ is right-continuous for $t \geq 0$,  positive for $t >0$, non-decreasing and satisfies the conditions
		$$g(0)=0, \quad g(\infty)=\lim_{t \to \infty}g(t)=\infty.$$
		\end{definition}By \cite[Chapter 1]{KR}, an N-function has also the following properties:
		\begin{enumerate}
			\item $G$ is continuous, convex, increasing, even and $G(0) = 0$.
	\item $G$ is super-linear at zero and at infinite, that is $$\lim_{x\rightarrow 0} \dfrac{G(x)}{x}=0$$and
	$$\lim_{x\rightarrow \infty} \dfrac{G(x)}{x}=\infty.$$
	
		\end{enumerate}Indeed, the above conditions serve as an equivalent definition of N-functions.
		
		An important property for N-functions is the following:
		\begin{definition}
		
			 We say that the N-function  $G$ satisfies the $\bigtriangleup_{2}$ condition if  there exists $C > 2$ such that
			\begin{equation*}
				G(2x) \leq C G(x) \,\,\text{~~~for all~~} x \in \mathbb{R}_+.
			\end{equation*}
			\end{definition}
			
			 Examples of functions satisfying the $\bigtriangleup_{2}$ condition are:
    \begin{itemize}
        \item $G(t)= t^{p}$, $t \geq 0$, $p > 1$;
        \item $G(t)=(1+|t|)\log(1+|t|) - |t|$;
        \item $G(t)=t^{p}\chi_{(0, 1]}(t) + t^{q}\chi_{(1, \infty)}(t)$, $t \geq 0$, $p, q > 1$.
    \end{itemize}

    
    By \cite[Theorem 4.1, Chapter 1]{KR},  an N-function  satisfies $\bigtriangleup_{2}$ condition if and only if there is $p^{+} > 1$ such that
			\begin{equation}\label{eq p mas}
		\frac{tg(t)}{G(t)} \leq p^{+}, ~~~~~\forall\, t>0.
	\end{equation}
	
	Associated to $G$ is  the N-function  complementary to it which is defined as follows:
		\begin{equation}\label{Gcomp}
			\widetilde{G} (t) := \sup \left\lbrace tw-G(w) \colon w>0 \right\rbrace .
		\end{equation}Moreover, the following representation holds for $\widetilde{G}$:
		$$\widetilde{G}(t)=\int_0^{t}g^{-1}(s)\,ds,$$where $g^{-1}$ is the right-continuous inverse of $g$. We recall that the role played by $\widetilde{G}$ is the same as the conjugate exponent functions when $G(t)=t^{p}$, $p >1$.
	
The definition of the complementary function assures that the following Young-type inequality holds
	\begin{equation}\label{2.5}
		at \leq G(t)+\widetilde{G} (a) \text{  for every } a,t \geq 0.
	\end{equation}
	
	We also quote the following useful lemma.
    \begin{lemma}\cite[Lemma 2.9]{BS}\label{G g}
        Let $G$ be an N-function. If $G$ satisfies \eqref{eq p mas} then 
        \[
            \tilde{G}(g(t)) \leq (p^+-1)G(t),
        \]
        where $g=G'$ and $\tilde{G}$ is the complementary function of $G.$
\end{lemma}
	
	By \cite[Theorem 4.3,   Chapter 1]{KR}, a necessary and sufficient condition for the N-function $\widetilde{G} $ complementary to $G$ to satisfy the $\bigtriangleup_{2}$ condition is that there is $p^{-} > 1$ such that
			\begin{equation}
	p^{-} \leq 	\frac{tg(t)}{G(t)}, ~~~~~\forall\, t>0.
	\end{equation}

	From now on, we will assume that the N-function  $G(t)= \int _{0} ^{t} g(\tau) d\tau$  satisfies the following growth behaviour:
	\begin{equation}\label{G1}
		1 < p^{-} \leq \frac{tg(t)}{G(t)} \leq p^{+} < \infty, ~~~~~\forall t>0.
	\end{equation}
	For the last section of this work, we need to assume the following condition
	\begin{equation}\label{L}
	p^- -1 \leq \frac{tg'(t)}{g(t)}\leq p^+-1, \quad \text{for } t>0.	
	\end{equation}
	It is well known the condition \eqref{L} implies \eqref{G1}.
	Given two $N$-functions $A$ and $B$, we say that $A$ increases essentially more slowly than $B$, denoted by
$A \ll B,$ if for any $c > 0$,
		\begin{equation*}
			\lim_{t \rightarrow \infty} \dfrac{A(ct)}{B(t)}=0.
		\end{equation*}

Given an N-function $G$, with $g= G'$, we define the Orlicz-Lebesgue class $L^G(\Omega)$ as follows
$$L^G(\Omega):=\left\lbrace u: \Omega \to \mathbb{R}, \int_\Omega G(u)\,dx < \infty\right\rbrace.$$If $G$ and its complementary function $\tilde{G}$ satisfy the $\Delta_2$ condition, then $L^G$ becomes a vector spaces. It is a Banach space with the Luxemburg norm
$$\|u\|_G:=\inf \left\lbrace \lambda > 0: \int_\Omega G\left(\dfrac{u}{\lambda}\right)\,dx \leq 1\right\rbrace.$$

Moreover, we denote the convex modular by
$$
\rho_{G}(u):=\int_\Omega G(u)\,dx.
$$

From now on we shall assume that any $N$-function and its complementary satisfy the $\Delta_2$-condition.

Moreover, for a positive integer $m$ we will also consider the next Orlicz-Sobolev spaces
$$W^{m, G}(\Omega):=\left\lbrace u \in L^G(\Omega), \,|D^\alpha u| \in L^G(\Omega), \,\text{for all multi-index }\, |\alpha|\leq m\right\rbrace,$$where for $\alpha=(\alpha_1,..., \alpha_n)$,
$$D^{\alpha}u=(\partial_1^{\alpha_1}u, ...,\partial_n^{\alpha_n}u)$$denotes the distributional derivative of $u$ of order $\alpha$. The space $W^{m, G}(\Omega)$ equipped with the norm
$$\|u\|:= \sum_{|\alpha|\leq m}\|D^{\alpha}u\|_G$$is a Banach space. 
We will always assume that $G$ satisfies:
\begin{itemize}
\item $$\int_0^1 \dfrac{G^{-1}(s)}{s^{(n+1)/n}}\,ds <\infty$$
\item $$\int_1^\infty \dfrac{G^{-1}(s)}{s^{(n+1)/n}}\,ds =\infty$$
\end{itemize}

 For a given $N$-function $G$, define the first order Sobolev conjugate function $G_1^{*}$  of $G$ by means of
 $$(G_1^{*})^{-1}(t):=\int_0^{t}\dfrac{G^{-1}(s)}{s^{1+1/n}}\,ds.$$Then $G_1^*$ is an $N$-function (see \cite{BSV2}). Next, we define the $m$-th order conjugate Sobolev function of $G$ as recursively as follows
 $$G^{*}_0 := G$$
 $$G^{*}_j:= (G^{*}_{j-1})^*, \quad j=1,..., m.$$At each stage, we assume that
 $$\int_0^1 \dfrac{(G_j^*)^{-1}(s)}{s^{(n+1)/n}}\,ds <\infty.$$We obtain in this way a finite sequence of $N$-functions $G^{*}_j$, $j=0, ..., m_0$, where $m_0$ is such that
 $$\int_1^\infty \dfrac{(G_{m_0-1}^*)^{-1}(s)}{s^{(n+1)/n}}\,ds =\infty$$but
 $$\int_1^\infty \dfrac{(G_{m_0}^*)^{-1}(s)}{s^{(n+1)/n}}\,ds <\infty.$$Indeed, $m_0 \leq n$, since by induction it can be proved that (see \cite{DT})
 $$(G_m^*)^{-1}(t)\leq K_m t^{(n-j)/n}.$$

 Then, we have the following embedding theorem for higher-order Orlicz-Sobolev spaces stated in \cite{DT}.

 \begin{theorem}\label{compact theorm}
 Let $\Omega \subset \mathbb{R}^{n}$ be a bounded domain with the cone property. Let $G$ be an $N$-function and let $m_0$ be defined as before. Then
 \begin{enumerate}
 \item if $1 \leq m \leq m_0$, then $W^{m, G}(\Omega) \hookrightarrow L^{G_{m}^{*}}(\Omega)$. Moreover, if $B$ is an $N$-function increasing essentially more slowly than $G^*_m$ near infinity, then the embedding $W^{m, G}(\Omega) \hookrightarrow L^{B}(\Omega)$ is compact;
 
 \item if $m > m_0$, then $W^{m, G}(\Omega) \hookrightarrow C(\Omega) \cap L^{\infty}(\Omega)$.
 \end{enumerate}
 \end{theorem}

In this paper we will consider the second order case $m=2$. Indeed, in order to take into account  boundary conditions, we denote $W^{2, G}_0(\Omega)$ the closure of $C_0^{\infty}(\Omega)$ in $W^{2, G}(\Omega)$. 
\begin{remark}\label{norma}
By \cite{WW},  the norm $\|u\|$ in $W^{2, G}_0(\Omega)$ is equivalent to
$$\|u\|_{2, G}:=\|\Delta u\|_G.$$
\end{remark}
From now, we will consider the norm $\|\cdot\|_{2, G}$. 
The relevant modular defined on $W^{2, G}_0(\Omega)$ is given by
$$\rho_{2,G}(u):=\rho_{G}(\Delta u)=\int_\Omega G(\Delta u)\,dx.$$By standard properties of modulars in Orlicz spaces, we obtain the following result.
\begin{proposition}Let $u \in W^{2, G}_0(\Omega)$. Then
\begin{itemize}
\item $\|u\|_{2, G} < 1 \Leftrightarrow \rho_{2, G}(u)<1.$
\item $\|u\|_{2, G} = 1 \Leftrightarrow \rho_{2, G}(u)=1.$
\item $\|u\|_{2, G} > 1 \Leftrightarrow \rho_{2, G}(u)>1.$
\item $\|u\|_{2, G} \to 0 \,(\to \infty) \Leftrightarrow \rho_{2, G}(u)\to 0\,(\to \infty)$
\end{itemize}
\end{proposition}

To close the section, we quote the following further  useful relation between modulars and norms.
\begin{lemma}\label{comp norm modular}
        Let $G$ be an N-function satisfying \eqref{G1},  and let 
        $\xi^\pm\colon[0,\infty)$ $\to\mathbb{R}$ be defined as
        \[
            \xi^{-}(t):= 
            \min \big \{  t^{p^{-}}, t^{p^{+}} \big  \} ,
            \quad \text{ and }  \quad
            \xi^{+}(t):=\max \big \{  t^{p^{-}}, t^{p^{+}} \big \} . 
         \] 
         Then
   
             $$\xi^{-}(\|u\|_{2, G}) \leq \rho_{2,G}( u) 
                    \leq  \xi^{+}(\|u\|_{2, G}).$$
    \end{lemma}

\section{On the definition of weak solutions for the biharmonic $g$-Laplacian}\label{weak}

In this section, we discuss the notion of weak solutions to problems of the form:
\begin{equation}\label{main problem}
\begin{cases}
\Delta_g^2 u = f(x, u)\quad \text{ in }\Omega\\
u\in W^{2, G}_0(\Omega),
\end{cases}
\end{equation}where the nonlinearity $f:\Omega \times \mathbb{R}\to \mathbb{R}$ is a Carath\'eodory function satisfying a growth condition so that
$$f(\cdot, u)v\in L^1(\Omega),$$for any $u, v\in W_0^{2, G}(\Omega)$. For instance, it is enough to assume that
$$|f(x, u)|\leq Cg(u)+ \phi(x), \quad g=G', \phi\in L^{\tilde{G}}(\Omega).$$

We consider the energy function $\Phi: W^{2, G}_0(\Omega) \to \mathbb{R}$ associated to the problem \eqref{main problem}:
$$\Phi(u):= \int_\Omega G(\Delta u)\,dx -\int_\Omega F(x, u)\,dx,\quad u \in W^{2, G}_0(\Omega),$$where
$$F(x, u):=\int_0^u f(x, t)\,dt.$$ 
Observe that $\Phi \in C^1(W^{2, G}_0(\Omega), \mathbb{R})$ and that
$$\left\langle \Phi'(u), v\right\rangle = \int_\Omega\dfrac{ g(\Delta u)}{|\Delta u|}\Delta u\Delta v\,dx-\int_\Omega f(x, u)v\,dx, \quad \text{for all }v\in W^{2, G}_0(\Omega).$$

In the next results, we will stablish some useful  properties of $L'$.

To motivate the definition of weak solutions and the boundary conditions in \eqref{main problem}, suppose that $u\in C^2_0(\Omega)$ is a  classical solution to Problem \eqref{main problem}. Then, for any $v \in C^2_0(\Omega)$, we have by applying integration by parts twice:
\begin{equation}\label{motivation}
\begin{split}
\int_\Omega f(x, u)v\,dx& = \int_\Omega v\Delta\left(\frac{g(|\Delta u|)}{|\Delta u|}\Delta u\right)\, dx \\& = \int_{\partial \Omega}v\frac{\partial}{\partial \nu}\left(\frac{g(|\Delta u|)}{|\Delta u|}\Delta u\right)\,d\sigma -\int_\Omega \nabla v \cdot \nabla\left(\frac{g(|\Delta u|)}{|\Delta u|}\Delta u\right)\,dx\\ & = \int_\Omega \Delta v\, \frac{g(|\Delta u|)}{|\Delta u|}\Delta u\,dx -\int_{\partial \Omega}\frac{g(|\Delta u|)}{|\Delta u|}\Delta u\frac{\partial v}{\partial \nu}\,d\sigma\\& = \int_\Omega \Delta v\,\frac{g(|\Delta u|)}{|\Delta u|}\Delta u \,dx.
\end{split}
\end{equation}

Based on the previous comments, we define next the notion of weak solutions to Problem \eqref{main problem}. 

\begin{definition}
We say that $u\in  W^{2, G}_0(\Omega)$ is a weak solution of Problem \eqref{main problem} if and only if
$$\int_\Omega \Delta v\,\frac{g(|\Delta u|)}{|\Delta u|}\Delta u \,dx = \int_\Omega f(x, u)v\,dx, \quad \text{for all }v\in  W^{2, G}_0(\Omega).$$
\end{definition}

\section{Basic properties of $\Delta_g^2$}\label{prop}

Let $L: W^{2, G}_0(\Omega) \to \mathbb{R}$ be defined by
$$L(u):=\int_\Omega G(\Delta u)\,dx.$$
\begin{proposition}\label{properties biharmonic}We have that $L \in C^1(W^{2, G}_0(\Omega),\mathbb{R})$ and its derivative $L'$ satisfies
\begin{itemize}
\item[(i)]  $L': W^{2, G}_0(\Omega)\to [W^{2, G}_0(\Omega)]'$ is a bounded, uniformly monotone  homeomorphism. In particular, the inverse operator $(L')^{-1}: [W^{2, G}_0(\Omega)]' \to W^{2, G}_0(\Omega)$ exists and it is continuous.
\item[(ii)] $L'$ is an operator of class $S_+$, that is, for any sequence $u_n \rightharpoonup u$ such that
$$\limsup_{n \to \infty}\left\langle L'(u)-L'(u_n), u-u_n\right\rangle \leq 0,$$there holds $u_n\to u$ in $W^{2, G}_0(\Omega)$.
\end{itemize}

\end{proposition}
\begin{proof}It is clear that $L': W^{2, G}_0(\Omega)\to [W^{2, G}_0(\Omega)]'$ given by
$$\left\langle L'(u), v\right\rangle = \int_\Omega \frac{g(|\Delta u|)}{|\Delta u|}\Delta u \Delta v\,dx$$is continuous in $W^{2, G}_0(\Omega)$.
We next prove that $L'$ is bounded, that is, it takes bounded subset of $W^{2, G}_0(\Omega)$ into bounded subsets of $[W^{2, G}_0(\Omega)]'$. Suppose for simplicity that $\|u\|_{2, G}\leq 1$, then for all $v\in W^{2, G}_0(\Omega)$ with $\|v\|_{2, G}\leq 1$ we have
\begin{equation*}
\begin{split}
|\left\langle L'(u), v\right\rangle|& = \big|\int_\Omega \frac{g(|\Delta u|)}{|\Delta u|}\Delta u \Delta v\,dx\big| \\& \leq C\left\|\frac{g(|\Delta u|)}{|\Delta u|}\Delta u \right\|_{\tilde{G}}\|\Delta v\|_{G}\\& \leq \max\left\lbrace \left[\rho_{\tilde{G}}\left(\frac{g(|\Delta u|)}{|\Delta u|}\Delta u \right)\right]^{1/\tilde{p}^+}, \left[\rho_{\tilde{G}}\left(\frac{g(|\Delta u|)}{|\Delta u|}\Delta u \right)\right]^{1/\tilde{p}^-}\right\rbrace\\&\leq C\max\left\lbrace [\rho_{G}(\Delta u)]^{1/\tilde{p}^+}, [\rho_{G}(\Delta u)]^{1/\tilde{p}^-}\right\rbrace \leq C.
\end{split}
\end{equation*}Hence, $\|L'(u)\|_{ [W^{2, G}_0(\Omega)]'}\leq C$ for all $\|u\|_{2, G}\leq 1$. 

To prove that $L'$ is uniformly monotone, we proceed as follows
\begin{equation*}
\begin{split}
\left\langle L'(u-v), u-v\right\rangle & = \int_\Omega g(|\Delta (u-v)|)|\Delta (u-v)|\,dx \\ & \geq p^{-}\int_\Omega G(\Delta (u-v))\,dx \\&\geq p^{-}\min\left\lbrace \|u-v\|^{p^{-}}_{2, G}, \|u-v\|^{p^{+}}_{2, G} \right\rbrace \\& \geq p^{-}\|u-v\|_{2, G}a(\|u-v\|_{2, G}),
\end{split}
\end{equation*}where
$$a(t):=\min\left\lbrace t^{p^{+}-1}, t^{p^{-}-1}\right\rbrace.$$Hence, $L'$ is uniformly monotone. 

In order to prove that $L'$ is a homeomorphism of class $S_{+}$, we will show that it is hemicontinuous, coercitive and apply  \cite[Theorem 26.A]{Z}. To prove that $L'$ is hemicontinuous, that is, the function 
$$t \to \left\langle L'(u+tv), w\right\rangle, \quad t \in [0, 1], \, u, v, w\in  W^{2, G}_0(\Omega)$$ is continuous, observe that for any $t\in [0, 1]$ and any sequence $t_n\in [0, 1]$ converging to $t_n$, it follows that
$$\frac{g(|\Delta u+t_n\Delta v|)}{|\Delta u+t_n\Delta v|}(\Delta u +t_n\Delta v)\to   \frac{g(|\Delta u+t\Delta v|)}{|\Delta u+t\Delta v|}(\Delta u +t\Delta v), \quad \text{as }n\to \infty$$and
\begin{align*}
\left|\frac{g(|\Delta u+t_n\Delta v|)}{|\Delta u+t_n\Delta v|}(\Delta u +t_n\Delta v)\Delta w\right|&\leq \tilde{G}\left(\frac{g(|\Delta u+t_n\Delta v|)}{|\Delta u+t_n\Delta v|}(\Delta u +t_n\Delta v)\right) + G(\Delta w)\\& \leq C(G(\Delta u)+G(\Delta v)+ G(\Delta w)) \in L^1(\Omega).
\end{align*}By dominated convergence theorem, we get that
$$ \left\langle L'(u+t_nv), w\right\rangle \to \left\langle L'(u+tv), w\right\rangle.$$We finally prove that $L'$ is coercitive, that is,
$$\dfrac{\left\langle L'(u), u\right\rangle}{\|u\|_{2, G}} \to \infty, \quad \text{as }\|u\|_{2, G}\to \infty.$$This is a consequence of the following inequalities for $\|u\|_{2, G}$ large
\begin{equation}
\left\langle L'(u), u\right\rangle \geq p^{-}\rho_G(\Delta u)\geq p^{-} \|u\|^{p^{-}}_{2, G}.
\end{equation}Therefore, by \cite[Theorem 26.A]{Z}, the proposition follows. 
\end{proof}

\section{Eigenvalue problem for the biharmonic $g$-Laplacian}\label{g}

In this section, we consider the following eigenvalue problem: given $r > 0$, find a function $u_r$ and $\lambda_r\in \mathbb{R}$ satisfying
\begin{equation}\label{eigenvalue problem}
\begin{cases}\Delta_g^2 u_r = \lambda_r \dfrac{g(|u_r|)}{|u_r|}u_r,\qquad  \text{in }\Omega\\
u_r\in  W^{2, G}_0(\Omega)\\
\int_\Omega G(u_r)\,dx = r.\end{cases} 
\end{equation}As the biharmonic $g$-Laplacian is not homogeneous, the eigenvalues and eigenfunctions depend on the normallization condition
$$ \int_\Omega G(u_r)\,dx = r.$$

\begin{theorem}
Let $G$ be an $N$-function satisfying \eqref{G1}. Given $r> 0$, let
\begin{equation}\label{inf problem}
c_r:=\inf\left\lbrace \int_\Omega G(\Delta u)\,dx: u\in  W^{2, G}_0(\Omega), \int_\Omega G(u)\,dx=r \right\rbrace.
\end{equation}
Then, there exists $u_r\in  W^{2, G}_0(\Omega)$ such that:
$$\int_\Omega G(\Delta u)\,dx= c_r.$$
\end{theorem}
\begin{proof}
Take a minimizing sequence $u_n\in  W^{2, G}_0(\Omega)$ for $c_r$, that is:
$$\int_\Omega G(u_n)\,dx=r, \quad \int_\Omega G(\Delta u_n)\,dx\to c_r,$$as $n\to \infty$. Hence
$$ \int_\Omega G(\Delta u_n)\,dx \leq C, \quad \text{for all n}$$and so by Lemma \ref{comp norm modular}, $u_n$ is bounded in $ W^{2, G}_0(\Omega)$. Hence, there is $u_r\in  W^{2, G}_0(\Omega)$ and a subsequence of $u_n$, still denoted by $u_n$, such that
$$u_n\rightharpoonup u_r\qquad\mbox{ in } W^{2, G}_0(\Omega)$$In particular, $\Delta u_n \rightharpoonup \Delta u_r$ in $L^G(\Omega)$. Since the modular $\rho_G$ is sequentially lower semi-continuous, we get
\begin{equation*}
\int_\Omega G(\Delta u_r)\,dx\leq \liminf_{n\to \infty}\int_\Omega G(\Delta u_n)\,dx=c_r.
\end{equation*}Hence, to conclude the proof, we just need to show that $u_r$ satisfies the constrain 
$$\int_\Omega G(u_r)\,dx=r.$$
By the compact Theorem \ref{compact theorm}, $u_n \to u_r$ strongly in $L^G(\Omega)$. Hence,
\begin{equation}\label{seq lower}
\begin{split}
\bigg|\int_\Omega (G(u_r)-G(u_n))\,dx\bigg| &= \bigg|\int_\Omega \int_0^1 g(tu_r+(1-t)u_n)(u_r-u_n)\,dt\,dx \bigg| \\ &
\leq \int_\Omega g(|u_r|+|u_n|)|u_r-u_n|\,dx \bigg|\\ & \leq C\|g(|u_r| + |u_n|)\|_{\tilde{G}}\|u_r-u_n\|_G \\ & \leq C\|u_r-u_n\|_G \to 0, \quad \text{as }n\to \infty.
\end{split}
\end{equation}Then, $u_r$ satisfies the normalization condition and so it solves \eqref{inf problem}. 
\end{proof}

\begin{remark}\label{remark positivity of ur}
Observe that $c_r > 0$, since if $c_r=0$, then
$$\Delta u_r= 0 \text{ in }\Omega,$$and so
$$\|u_r\|_{2, G}=\|\Delta u_r\|_G=0$$which yields $u_r=0$. This clearly contradicts the constraint
$$\int_\Omega G(u_r)\,dx=r>0.$$ 
\end{remark}

We next show that $u_r$ is indeed a solution of the eigenvalue problem \eqref{eigenvalue problem}. Instead of applying Lagrange Multipliers, we employ a method from \cite{MT} which can be applied even if the $N$-functions do not satisfy the $\Delta_2$-condition. 

We start by quoting the following technical lemma from \cite{MT}.

\begin{lemma}\label{auxiliary lemma}
Let $u, v \in L^G(\Omega)$ such that 
$$\int_\Omega \dfrac{g(|u|)}{|u|}uv\,dx \neq 0.$$Then, the condition
$$\int_\Omega G((1-\varepsilon)u+\delta v)\,dx=\int_\Omega G(u)\,dx,$$defines a continuously differentiable function $\delta=\delta(\varepsilon)$ in some interval $(-\varepsilon_0, \varepsilon_0)$, with $\varepsilon_0
>0$. Moreover, $\delta(0)=0$ and
$$\delta'(0)= \dfrac{\int_\Omega \dfrac{g(|u|)}{|u|}uu\,dx}{\int_\Omega \dfrac{g(|u|)}{|u|}uv\,dx}.$$
\end{lemma}

Finally, we state that $u_r$ solves \eqref{eigenvalue problem}.

\begin{theorem}
Let $u_r$ be a solution of \eqref{inf problem}. Then, there is $\lambda_r>0$ such that
\begin{equation}\label{eq objetivo}
\Delta_g^2 u_r = \lambda_r \dfrac{g(|u_r|)}{|u_r|}u_r\quad \text{weakly in }\Omega.
\end{equation}

\end{theorem}

\begin{proof}
Define linear functionals $F:  W^{2, G}_0(\Omega)\to \mathbb{R}$ and $G:  W^{2, G}_0(\Omega)\to \mathbb{R}$ by:
$$F(v)= \int_\Omega \dfrac{g(|u_r|)}{|u_r|}u_rv\,dx,$$
$$H(v)= \int_\Omega \dfrac{g(|\Delta u_r|)}{|\Delta u_r|}\Delta u_r \Delta v\,dx.$$We shall prove that Ker H $\subset $ Ker F, which will show by \cite[Proposition 43.1]{Z} that there is $\lambda_r \in \mathbb{R}$ such that \eqref{eq objetivo} holds. Let 
$$P_F:= \left\lbrace v\in  W^{2, G}_0(\Omega): F(v)>0\right\rbrace$$and
$$P_H:= \left\lbrace v\in  W^{2, G}_0(\Omega): H(v)>0\right\rbrace.$$We will prove that $P_F\subset P_H$. Let $v\in P_F$. Then, 
$$\int_\Omega \dfrac{g(|u_r|)}{|u_r|}u_rv\,dx\neq 0.$$From Lemma \ref{auxiliary lemma}, there are $\varepsilon_0> 0$ and a $C^1((-\varepsilon_0, \varepsilon_0))$-function $\delta=\delta(\varepsilon)$ such that
$$\int_\Omega G((1-\varepsilon)u_r+\delta v)\,dx = r, \quad \text{for all }\varepsilon\in (-\varepsilon_0, \varepsilon_0).$$By Lemma \ref{auxiliary lemma}, we get $\delta'(0)> 0$, so 
$$\dfrac{1}{2}\delta'(0)< \delta'(\varepsilon)<2\delta'(0),$$for all $\varepsilon$ small. Hence, for $\varepsilon> 0$ small enough,
\begin{equation}\label{ineq delta epsilon}
\dfrac{1}{2}\delta'(0)<\dfrac{\delta(\varepsilon)}{\varepsilon} <2\delta'(0).
\end{equation}Next, denote $v_\varepsilon= (1-\varepsilon)u_r+\delta(\varepsilon)v$. Since $u_r$ solves \eqref{inf problem}, there holds
\begin{equation}\label{dominated 1}
\int_\Omega \dfrac{G(|\Delta v_\varepsilon|)-G(|\Delta u_r|)}{\delta(\varepsilon)}\,dx \geq 0.
\end{equation}Also, 
\begin{equation*}
\begin{split}
I_\varepsilon:=\dfrac{G(|\Delta v_\varepsilon|)-G(|\Delta u_r|)}{\delta(\varepsilon)}& = \dfrac{G(|\Delta v_\varepsilon|)-G(|\Delta u_r|)}{|\Delta v_\varepsilon|-|\Delta u_r|}\dfrac{|\Delta v_\varepsilon|^2-|\Delta u_r|^2}{(|\Delta v_\varepsilon|+|\Delta u_r|)\delta(\varepsilon)}\\& =\dfrac{G(|\Delta v_\varepsilon|)-G(|\Delta u_r|)}{|\Delta v_\varepsilon|-|\Delta u_r|}\dfrac{|\Delta v_\varepsilon|-|\Delta u_r|}{\varepsilon}\dfrac{\varepsilon}{\delta(\varepsilon)}
\end{split}
\end{equation*}When $\varepsilon\to 0^+$, since $\Delta v_\varepsilon\to \Delta u_r$ a.e. in $\Omega$, it follows
\begin{equation*}
\dfrac{G(|\Delta v_\varepsilon|)-G(|\Delta u_r|)}{|\Delta v_\varepsilon|-|\Delta u_r|}\to  g(|\Delta u_r|).
\end{equation*}Moreover,
$$\dfrac{\varepsilon}{\delta(\varepsilon)}\to \dfrac{1}{\delta'(0)}, \text{ as }\varepsilon\to 0^{+}.$$Finally, letting
$$\Phi(\varepsilon)=|(1-\varepsilon)\Delta u_r+\delta(\varepsilon)v|,$$we get
$$\dfrac{|\Delta v_\varepsilon|-|\Delta u_r|}{\varepsilon}=\dfrac{\Phi(\varepsilon)-\Phi(0)}{\varepsilon}\to \Phi'(0)=\dfrac{\Delta u_r}{|\Delta u_r|}\left(-\Delta u_r + \delta'(0)\Delta v \right).$$Therefore,
\begin{equation}\label{dominated 2}
I_\varepsilon \to g(|\Delta u_r|)\left(-\dfrac{1}{\delta'(0)}|\Delta u_r| + \dfrac{\Delta u_r \Delta v}{|\Delta u_r|} \right),
\end{equation}a.e. in $\Omega$ when $\varepsilon\to 0^+$. In addition,
\begin{equation*}
\begin{split}
\bigg|\dfrac{G(|\Delta v_\varepsilon|)-G(|\Delta u_r|)}{\delta(\varepsilon)}\bigg|& \leq \left(g(|\Delta v_\varepsilon|)+g(|\Delta u_r|) \right)\dfrac{|\Delta v_\varepsilon-\Delta u_r|}{\delta(\varepsilon)}\\ & \leq  \left(g(|\Delta u_r| +\|\delta(\varepsilon)\|_{L^{\infty}([-\varepsilon_0/2, \varepsilon_0/2])}|\Delta v| \right)\left(\dfrac{\varepsilon}{\delta(\varepsilon)}|\Delta u_r| + |\Delta v| \right)\\ & \leq  \left(g(|\Delta u_r| +\|\delta(\varepsilon)\|_{L^{\infty}([-\varepsilon_0/2, \varepsilon_0/2])}|\Delta v| \right)\left(\dfrac{2}{\delta'(0)}|\Delta u_r| + |\Delta v| \right)\in L^1(\Omega),
\end{split}
\end{equation*}where we have used \eqref{ineq delta epsilon} and H\"{o}lder's inequality. Therefore, by dominated convergence theorem, we conclude taking the limit as $\varepsilon\to 0^+$ in \eqref{dominated 1} and recalling \eqref{dominated 2}, that
\begin{equation*}
\int_\Omega \dfrac{g(|\Delta u_r|)}{|\Delta u_r|}\Delta u_r \Delta v\,dx\geq \dfrac{1}{\delta'(0)}\int_\Omega g(|\Delta u_r|)|\Delta u_r|\,dx > 0.
\end{equation*}Observe that the last inequality follows from Remark \ref{remark positivity of ur}. Hence, $v\in P_H$. This concludes the proof. 

\end{proof}

The following result gives a straightforward lower bound of $\lambda_r$ in terms of $c_r$, using the condition \eqref{G1}.
\begin{proposition}
For any $r>0$, the corresponding eigenvalue $\lambda_r$ satisfies the lower bound
$$\lambda_r \geq \dfrac{p^- c_r}{rp^+}.$$
\end{proposition}
 \begin{proof}
 To prove the proposition, take $v=u_r$ in the definition of weak solution to \eqref{eigenvalue problem}. Then,
 $$\int_\Omega g(|\Delta u_r|)|\Delta u_r|\,dx=\lambda_r\int_\Omega g(| u_r|)| u_r|\,dx.$$By  \eqref{G1}, we get
 $$p^{-}\int_\Omega G(\Delta u_r)\,dx \leq  \int_\Omega g(|\Delta u_r|)|\Delta u_r|\,dx =\lambda_r\int_\Omega g(| u_r|)| u_r|\,dx \leq \lambda_r p^+\int_\Omega G(u_r)\,dx.$$The proof follows by recalling that $\int_\Omega G(u_r)\,dx=r$ and the definition of $c_r$ \eqref{inf problem}.
 \end{proof}

The objective of the next result is to bound from below the eigenvalues of $\Delta^2_g$ in terms of the variational quantity:
$$\lambda_0:=\inf_{u\in  W^{2, G}_0(\Omega), u\neq 0}\dfrac{\int_\Omega g(|\Delta u|)|\Delta u|\,dx}{\int_\Omega g(|u|)|u|\,dx}.$$

\begin{proposition}
Any eigenvalue $\lambda_r$ of \eqref{eigenvalue problem} satisfies
$$\lambda_r \in [\lambda_0, \infty).$$
\end{proposition}
\begin{proof}
Suppose that
\begin{equation}\label{ineq lambdas}
\lambda_r< \lambda_0
\end{equation}for some $r> 0.$ Then, we know that there is $u_r\in  W^{2, G}_0(\Omega)$, $\int_\Omega G(u_r)\,dx=1$, such that
$$\int_\Omega \dfrac{g(|\Delta u_r|)}{|\Delta u_r|}\Delta u_r\Delta v\,dx = \lambda_r\int_\Omega\dfrac{g(| u_r|)}{| u_r|} u_r v\,dx,$$for all $v\in  W^{2, G}_0(\Omega)$. In particular, taking $v=u_r$, we get
$$\int_\Omega g(|\Delta u_r|)|\Delta u_r|\,dx= \lambda_r\int_\Omega g(|u_r|)|u_r|\,dx$$which, in view of \eqref{ineq lambdas}, contradicts the definition of $\lambda_0$. 
\end{proof}

\section{Nonlinear eigenvalue problem with two Orlicz functions}\label{b}
Here, we are concerned with an eigenvalue problem driven by two $N$-functions $G$ and $B$. We will see that depending on the relative growth of $G$ and $B$, there is a continuous spectrum concentrated around $0$, around $\infty$ or that coincides with the whole interval $(0, \infty)$.

In any case, we let $G$ and $B$ be two $N$-functions satisfying \eqref{G1}, and we consider the following nonlinear eigenvalue problem
\begin{equation}\label{problem eigenvalue}
\begin{cases}\Delta_g^2 u = \lambda \dfrac{b(|u|)}{|u|}u,\qquad  \text{in }\Omega\\
u\in  W^{2, G}_0(\Omega),
\end{cases} 
\end{equation}where $b=B'$. Given the $\Delta_2$-exponents:
$$p^-:=\inf_{t>0}\dfrac{tg(t)}{G(t)},\,\, p^+:=\sup_{t>0}\dfrac{tg(t)}{G(t)},$$and the corresponding $p^{-}_B$ and $p_B^{+}$ for $B$, we will consider all relative growth behaviours  between $G$ and $B$ in terms of the above exponents, that is:
\begin{itemize}
\item Theorem \ref{eigenvalue}: $1<p_B^{-}<p^{-}<p_B^{+}<p^{+}$ and $1<p_B^{-}<p^{-}<p^{+}<p_B^{+}$;
\item Theorem \ref{eigenvalueinfty}: $1<p_B^{-}<p_B^{+}<p^{-}<p^{+}$;
\item Theorem \ref{eigenvalueinfty2}: $p^{-}<p_B^{-}<p^{+}<p_B^{+}$ and $p^{-}<p_B^{-}<p_B^{+}<p^{+}$;
\item Theorem  \ref{eigenvalue3}: $p^{-}<p^{+}<p_B^{-}<p_B^{+}$.
\end{itemize}

\begin{theorem}[Spectrum concentrated around $0$]\label{eigenvalue}
Let $G$ and $B$ be two $N$-functions satisfying \eqref{G1}. Moreover assume that $B \ll G^*$, that
\begin{equation}\label{hyp b}
1<p_B^{-}<p^{-}<p_B^{+},
\end{equation}and that there is $t_0>0$ such that
\begin{equation}\label{assump exponent}
p_B^{-}-1 \leq \dfrac{tb'(t)}{b(t)}, \quad \text{for all }t\in (0, t_0), \, b= B'.
\end{equation}Then, there is $\lambda^{*}>0$ such that any $\lambda\in (0, \lambda^{*})$ is an eigenvalue of the problem \eqref{problem eigenvalue}.
\end{theorem}

\begin{remark}Observe that due to the hypothesis \eqref{hyp b} on the exponents of $B$, we cannot guarantee that the right hand side term
$$\lambda  \dfrac{b(|u|)}{|u|}u$$satisfies the Ambrosetti-Rabinowitz condition (that is, there is $\mu > 2$ such that $\mu B(u)\leq ub(u)$). Then, although the associated energy functional has a mountain pass geometry, the Palais Smale condition may not hold. So, the mountain pass Theorem may not be possible to use in this context.

\end{remark}

Before giving the proof of Theorem \ref{eigenvalue}, we will prove some preliminary results. For any $\lambda> 0$, let define $\Phi_\lambda: W^{2, G}_0(\Omega)\to \mathbb{R}$ by
$$\Phi_\lambda(u):=\int_\Omega G(\Delta u)\,dx -\lambda\int_\Omega B(u)\,dx.$$

\begin{lemma}\label{lemaa}
Under the assumptions of Theorem \ref{eigenvalue}, there is $\lambda^{*}>0$ such that for any $\lambda\in (0, \lambda^{*})$, there are $\rho, \alpha>0$ such that
$$\Phi_\lambda(u)\geq \alpha, \quad\text{for any }u\in \partial B_\rho(0).$$
\end{lemma}

\begin{proof}
By the compact embedding Theorem \ref{compact theorm}, there is $C>1$ such that
$$||u||_B\leq C||u||_{2, G}.$$Choose $0<\rho < 1$ so that
\begin{equation}\label{cond rho}
C\rho <1.
\end{equation}Let $\|u\|_{2, G}=\rho$, then
\begin{equation}\label{large calc}
\begin{split}
\Phi_\lambda(u)&=\int_\Omega G(\Delta u)\,dx -\lambda\int_\Omega B(u)\\ & \geq \min\left\lbrace \|u\|_{2, G}^{p^{+}},\|u\|_{2, G}^{p^{-}}\right\rbrace -\lambda \max\left\lbrace \|u\|_B^{p_B^{+}},\|u\|_B^{p_B^{-}}\right\rbrace \\& \geq \min\left\lbrace \|u\|_{2, G}^{p^{+}},\|u\|_{2, G}^{p^{-}}\right\rbrace -\lambda  \max\left\lbrace (C\|u\|_{2, G})^{p_B^{+}},(C\|u\|_{2, G})^{p_B^{-}}\right\rbrace \\ & = \|u\|_{2, G}^{p^{+}}-\lambda C \|u\|_{2, G}^{p_B^{-}} \qquad (\text{by }\eqref{cond rho})\\ & = \rho^{p^{+}}-\lambda C \rho^{p_B^{-}}=\rho^{p_B^{-}}(\rho^{p^+-p_B^{-}}-\lambda C).
\end{split}
\end{equation}Hence, choosing 
$$\lambda^*=\dfrac{\rho^{p^+-p_B^{-}}}{2C}$$we conclude the proof from \eqref{large calc} and the assumption \eqref{assump exponent}.
\end{proof}

\begin{lemma}\label{lemb}
Under the assumptions of Theorem \ref{eigenvalue}, there is $v\in W^{2, G}_0(\Omega)$, $v\geq 0$, $v\neq 0$ such that
$$\Phi_\lambda(tv)< 0 \qquad \text{for all $t>0$ small}.$$
\end{lemma}

\begin{proof}First, observe that assumption \eqref{assump exponent} implies that the function
$$h(t):=\dfrac{tb(t)}{B(t)}, \quad t>0,$$is nondecreasing. Indeed, 
$$h'(t)= \dfrac{(b(t)+tb'(t))B(t)-tb(t)b(t)}{B(t)^2}\geq  \dfrac{(b(t)+tb'(t))B(t)-p_B^-B(t)b(t)}{B(t)^2}=\dfrac{tb'(t)+(1-p_B^{-})b(t)}{B(t)}>0.$$Hence, for $0<\varepsilon< p^{-}-p_B^{-}$, there is $0<t_0<1$ such that
$$\dfrac{tb(t)}{B(t)}< p_B^{-}+\varepsilon_0<p^{-},$$for any $0 <t<t_0$. Integrating both side and assuming that $B(1)=1$ for simplicity, we get
\begin{equation}\label{Bop}
B(t)\geq t^{p_B^{-}+\varepsilon}, \quad t \in (0, t_0).
\end{equation} Let now $v\in C_0^{\infty}(\Omega)$, nonnegative and such that $v(x), |\Delta v(x)|\leq 1/t_0$ for all $x\in \Omega$. Then, for any $t\in (0, t_0)$, it follows
\begin{equation*}
\begin{split}
\Phi_\lambda(tv) &\leq t^{p^{-}}\max\left\lbrace \|v\|_{2, G}^{p^{+}},\|v\|_{2, G}^{p^{-}}\right\rbrace -\lambda t^{p^{-}_B+\varepsilon}\|v\|_{p_B^{-}+\varepsilon}^{p^{-}_B+\varepsilon}\\ & = c_1(v) t^{p^{-}} -\lambda t^{p^{-}_B+\varepsilon} c_2(v)\\ & = t^{p^{-}}(c_1(v)-\lambda c_2(v)t^{p^{-}_B+\varepsilon-p^{-}}).
\end{split}
\end{equation*}Hence, taking $t$ small enough (and depending on the norms of $v$), we conclude the proof of the lemma. 

\end{proof}
\begin{remark}
 We point out that assumption \eqref{hyp b} means that $G$ grows much slower than $B$ near $0$. Indeed, as it will be shown in the proof of Lemma \ref{lemb}, hypothesis \eqref{assump exponent} implies that the quotient
$$\dfrac{tb(t)}{B(t)}$$is nodecreasing. Hence,  for $\varepsilon \in (0, p^{-}-p_B^{-})$, there is $0<t_0<1$ (see again the proof of Lemma \ref{lemb}) such that
$$B(t)\geq t^{p_B^{-}+\varepsilon}, \quad \text{for all }t\in (0, t_0).$$Therefore, by the choice of $\varepsilon$,
$$\lim_{t\to 0^{+}}\dfrac{G(t)}{B(t)} \leq \lim_{t\to 0^{+}}\dfrac{t^{p^{-}}}{t^{p_B^{-}+\varepsilon}}=0.$$This proves the assertion. Also, observe that we are not assuming any behaviour between  $B$ and $G$ around $\infty.$
\end{remark}
\begin{remark}
Observe that the lower bound $B(t)\geq \min\left\lbrace t^{p_B^-},t^{p_B^{+}} \right\rbrace$ does not help  to prove Lemma \ref{lemb}. We need to use a sharper lower bound, see \eqref{Bop} .
\end{remark}
Next, we finish the proof of the main theorem

\begin{proof}[Proof of Theorem \ref{eigenvalue}]
First, by Lemma \ref{lemaa}, there is $\rho > 0$ such that
$$\inf_{\partial B_\rho(0)}\Phi_\lambda> 0.$$Moreover, for any $u\in B_\rho(0)$,
\begin{equation}
\Phi_\lambda(u)\geq \|u\|_{2, G}^{p^{+}}-C\lambda\|u\|_{2, G}^{p_B^{-}}\geq -C\lambda \rho^{p_B^{-}} > -\infty.
\end{equation}Thus, 
\begin{equation}\label{d}
-\infty < \inf_{\overline{B_\rho(0)}}\Phi_\lambda < 0.
\end{equation}
Let
$$0<\varepsilon <\inf_{\partial B_\rho(0)}\Phi_\lambda - \inf_{\overline{B_\rho(0)}}\Phi_\lambda,$$then by the Ekeland's Variational Principle, there is $u_\varepsilon\in \overline{B_\rho(0)}$ such that
\begin{equation}\label{ue}
\Phi_\lambda(u_\varepsilon) < \inf_{\overline{B_\rho(0)}}\Phi_\lambda +\varepsilon,
\end{equation}
and
\begin{equation}\label{u epsilon ineq}
\Phi_\lambda(u_\varepsilon) < \Phi_\lambda(u)+\varepsilon\|u_\varepsilon-u\|_{2, G},\quad u\neq u_\varepsilon.
\end{equation}Since \eqref{d} and \eqref{ue}
$$\Phi_\lambda(u_\varepsilon) < \inf_{\overline{B_\rho(0)}}\Phi_\lambda +\varepsilon < \inf_{\partial B_\rho(0)}\Phi_\lambda,$$we deduce that $u_\varepsilon$ belongs to $B_\rho(0)$.\\ Let us define now the functional $\Psi_\lambda:\overline{B_\rho(0)} \to \mathbb{R}$ given by 
$$\Psi_\lambda(u):=\Phi_\lambda(u)+\varepsilon\|u-u_\varepsilon\|_{2, G}.$$Then, by \eqref{u epsilon ineq}, $u_\varepsilon$ is a minimum point of $\Psi_\lambda$ and so
\begin{equation}\label{tridente}
\dfrac{\Psi_\lambda(u_\varepsilon+tv)-\Psi_\lambda(u_\varepsilon)}{t}\geq 0, \quad t>0, v\in W^{2, G}_0(\Omega).
\end{equation}Letting $t\to 0^{+}$ in \eqref{tridente}, we obtain
$$\left\langle \Phi'_\lambda(u_\varepsilon), v\right\rangle +\varepsilon\|v\|_{2, G} > 0.$$This implies that $\|\Phi'(u_\varepsilon)\|\leq \varepsilon$. In this way, taking $\varepsilon=\frac{1}{n}$, we build a sequence $u_n \in B_\rho(0)$ such that
\begin{equation}\label{PS seq}
\Phi_\lambda(u_n)\to \inf_{\overline{B_\rho(0)}}\Phi_\lambda, \quad \Phi'_\lambda(u_n)\to 0.
\end{equation} Since $u_n$ is bounded, there is $u\in W^{2, G}_0(\Omega)$ such that $u_n\rightharpoonup u$ in $W^{2, G}_0(\Omega)$. Hence, by Theorem \ref{compact theorm},
$$u_n\to u\text{ in }L^B(\Omega).$$Next, observe that
\begin{equation}\label{S plus}
\begin{split}
& \int_\Omega \left(\dfrac{g(|\Delta u|)}{|\Delta u|}\Delta u-\dfrac{g(|\Delta u_n|)}{|\Delta u_n|}\Delta u_n \right)(\Delta u-\Delta u_n)\,dx \\& \qquad = \int_\Omega \left(\dfrac{g(|\Delta u|)}{|\Delta u|}\Delta u-\dfrac{g(|\Delta u_n|)}{|\Delta u_n|}\Delta u_n \right)(\Delta u-\Delta u_n)\,dx \\ &\qquad  -\lambda\int_\Omega  \left(\dfrac{b(| u|)}{| u|} u-\dfrac{b(| u_n|)}{| u_n|} u_n \right)( u- u_n)\,dx+ \lambda\int_\Omega  \left(\dfrac{b(| u|)}{| u|} u-\dfrac{b(| u_n|)}{| u_n|} u_n \right)( u- u_n)\,dx\\ & \qquad =\left\langle \Phi'_\lambda(u)-\Phi'_\lambda(u_n), u-u_n \right\rangle+ \lambda\int_\Omega  \left(\dfrac{b(| u|)}{| u|} u-\dfrac{b(| u_n|)}{| u_n|} u_n \right)( u- u_n)\,dx.
\end{split}
\end{equation}By \eqref{PS seq} and the weak convergence of $u_n$ to $u$, we obtain:
$$\left\langle \Phi'_\lambda(u)-\Phi'_\lambda(u_n), u-u_n \right\rangle \to 0,$$as $n\to \infty$. Moreover, 
$$\bigg|\int_\Omega  \left(\dfrac{b(| u|)}{| u|} u-\dfrac{b(| u_n|)}{| u_n|} u_n \right)( u- u_n)\,dx \bigg|\leq C\bigg\|\dfrac{b(| u|)}{| u|} u-\dfrac{b(| u_n|)}{| u_n|} u_n \bigg\|_{\tilde{B}}\|u_n-u\|_{B}\to 0$$as $n\to \infty$, since $$\bigg\|\dfrac{b(| u|)}{| u|} u-\dfrac{b(| u_n|)}{| u_n|} u_n \bigg\|_{\tilde{B}}\leq C$$for all $n$ by the weak convergence in $W^{2, G}_0(\Omega)$ and $\|u_n-u\|_{B}\to 0$ by the strong convergence in $L^B(\Omega).$ Therefore, by \eqref{S plus}, we obtain that
$$\limsup_{n\to \infty} \int_\Omega \left(\dfrac{g(|\Delta u|)}{|\Delta u|}\Delta u-\dfrac{g(|\Delta u_n|)}{|\Delta u_n|}\Delta u_n \right)(\Delta u-\Delta u_n)\,dx \leq 0,$$hence by Proposition \ref{properties biharmonic}, we conclude that $u_n \to u$ strongly in $W^{2, G}_0(\Omega)$. Therefore, \eqref{PS seq} implies that 
$$\left\langle \Phi'_\lambda(u), v\right\rangle =0,$$for any $v\in W^{2, G}_0(\Omega)$. So, $\lambda\in (0, \lambda^{*})$ is an eigenvalue of \eqref{problem eigenvalue}.
\end{proof}

\begin{remark}Observe that the critical point $u$ is not zero, since by \eqref{PS seq},
$$\Phi_\lambda(u) = d = \inf_{\overline{B_\rho(0)}}\Phi_\lambda <0.$$
\end{remark}

On the other hand, we get for any $\lambda\in (0, \lambda^{*})$, there exists $u_\lambda$ that
$$
\int_\Omega \dfrac{g(|\Delta u_\lambda|)}{|\Delta u_\lambda|}\Delta u_\lambda \Delta v\,dx=\lambda\int_\Omega \dfrac{b(|u_\lambda|)}{|u_\lambda|} u_\lambda v\,dx.
$$
Choosing $v=u_\lambda$, we have
$$
\frac{\int_\Omega g(|\Delta u_\lambda|)|\Delta u_\lambda|\,dx}{\int_\Omega b(|u_\lambda|)|u_\lambda|\,dx}=\lambda.
$$
Taking infimum
$$\inf_{u\in W_0^{2,G}(\Omega), u\neq 0}\frac{\int_\Omega g(|\Delta u|)|\Delta u|\,dx}{\int_\Omega b(|u|)|u|\,dx}\leq \lambda.$$

Hence, as  $\lambda\to 0^+$, we obtain
$$\inf_{u\in W_0^{2,G}(\Omega), u\neq 0}\dfrac{\int_\Omega g(|\Delta u|)|\Delta u|\,dx}{\int_\Omega b(|u|)|u|\,dx}=0,$$ and so the following anti-Sobolev type inequality holds:

\begin{corollary}
For any constant $C> 0$, there is $u\in W_0^{2,G}(\Omega)$ such that
$$\int_\Omega g(|\Delta u|)|\Delta u|\,dx\leq C\int_\Omega b(|u|)|u|\,dx.$$
\end{corollary}

Observe that in the previous case, we do not use any relation between the exponents $p^{-}$ and $p_B^+$. Indeed, the proof also works in the case
$$p_B^{-}<p_B^{+}<p^{-}.$$However, we will see that in this case, the spectrum is larger.

\begin{theorem}\label{eigenvalueinfty}
Let $G$ and $B$ be two $N$-functions satisfying \eqref{G1}. Moreover assume that $B \ll G^*$ and  that
\begin{equation}\label{hyp b 3}
1<p_B^{-}<p_B^{+}<p^-,
\end{equation}Then, any $\lambda > 0$ is an eigenvalue of the problem \eqref{problem eigenvalue}.
\end{theorem}

\begin{proof}
In this case, we will show that  for any $\lambda> 0$, the functional $\Phi_\lambda$ has a nontrivial minimum point. 

We start by showing that the functional   $\Phi_\lambda$ is coercive:
$$\Phi_\lambda(u)\to \infty, \text{ as }\|u\|_{2, G}\to \infty.$$Observe, by a similar reasoning as in \eqref{large calc}, that for $\|u\|_{2, G}> 1$ and thanks to \eqref{hyp b 3},
$$\Phi_\lambda(u) \geq \|u\|_{2, G}^{p^-}-C\lambda \|u\|_{2, G}^{p_B^+}\geq \|u\|_{2, G}^{p^-}(1-C\lambda \|u\|_{2, G}^{p_B^{+}-p^-}) \to \infty$$as $\|u\|_{2, G}\to \infty$. This shows that $\Phi_\lambda$ is coercive. Moreover, $\Phi_\lambda$ is sequentially weakly lower-semicontinuous, since the modular
$$\rho_G(\cdot)$$is seq. lower semi-continuous and 
$$\int_\Omega B(u_n)\,dx\to \int_\Omega B(u)\,dx$$for any sequence $u_n \rightharpoonup u$ in $W_0^{2,G}(\Omega)$ (see for instance the calculation \eqref{seq lower}). Then, there is $u\in W_0^{2,G}(\Omega)$ such that
$$\Phi_\lambda(u)=\inf_{W_0^{2,G}(\Omega)}\Phi_\lambda.$$Hence $u$ is a weak solution of \eqref{problem eigenvalue}. In order to show that $u \neq 0$, take any $v\in C_0^{\infty}(\Omega)$ so that $0<\|v\|_{2, G}$. Then,

$$\Phi_\lambda(v)\leq \max\left\lbrace \|v\|_{2, G}^{p^+}, \|v\|_{2, G}^{p^-}\right\rbrace-\lambda C \min\left\lbrace \|v\|_{2, G}^{p_B^+}, \|v\|_{2, G}^{p_B^-}\right\rbrace.$$Hence, choose
\begin{equation}\label{cond lambda1}
\lambda^{**}:=\dfrac{\max\left\lbrace \|v\|_{2, G}^{p^+}, \|v\|_{2, G}^{p^-}\right\rbrace}{C \min\left\lbrace \|v\|_{2, G}^{p_B^+}, \|v\|_{2, G}^{p_B^-}\right\rbrace}>0.
\end{equation}In this way, for any $\lambda > \lambda^{**}$, there is $v$ so that $\Phi_\lambda(v)<0$. Next, we will analize the possible values of $\lambda^{**}$. Introducing the function
$$F(t):= \dfrac{\max\left\lbrace t^{p^+}, t^{p^-}\right\rbrace}{C \min\left\lbrace t^{p_B^+}, t^{p_B^-}\right\rbrace},$$there holds
$$F(t)=\begin{cases} \dfrac{1}{C}t^{p^{-}-p_B^{+}}, \quad 0<t<1\\\dfrac{1}{C}t^{p^{+}-p_B^{-}}, \quad t>1.
\end{cases}$$By assumption \eqref{cond lambda1}, $p^{-}-p_B^{+}, p^{+}-p_B^{-}>0$, so the range of $F$ is $(0, \infty)$. Hence, by \eqref{cond lambda1},  any $\lambda>0$ may be chosen to get $\Phi_\lambda(v)<0$ for some $v$.
\normalcolor
We conclude that $u\neq 0$. This ends the proof of the Theorem. 
\end{proof}

\begin{remark}
Observe that under the assumption \eqref{hyp b 3}, we have
$$\dfrac{G(t)}{B(t)}\geq \dfrac{t^{p^{-}}}{t^{p_B^{+}}}\to \infty \quad \text{as }t\to \infty$$and
$$\dfrac{G(t)}{B(t)}\leq \dfrac{t^{p^{+}}}{t^{p_B^{-}}}\to 0 \quad \text{as }t\to 0.$$Thus, in Theorem \ref{eigenvalueinfty}, we have $B\ll G$ near $\infty$ and $G\ll B$ near $0$.
\end{remark}
\normalcolor


In the next result, we state the case where $B$ grows much slower than $G$ near $0$. 
\begin{theorem}[Spectrum concentrated at $\infty$]\label{eigenvalueinfty2}
Let $G$ and $B$ be two $N$-functions satisfying \eqref{G1}. Moreover assume that $B \ll G^*$ and that
\begin{equation}\label{hyp b 22}
1<p^{-}<p_B^{-}<p^{+}
\end{equation} and that there is $t_0>0$
\begin{equation}\label{assumt g prima}
p^{-}-1\leq \dfrac{tg'(t)}{g(t)}, \quad \text{for all }t\in (0, t_0).
\end{equation}Then, there is $\lambda^{**}>0$ such that any $\lambda\in (\lambda^{**},\infty)$ is an eigenvalue of the problem \eqref{problem eigenvalue}.
\end{theorem}

\begin{proof}
We combine the strategies of the proofs of Theorem \ref{eigenvalue} and Theorem \ref{eigenvalueinfty}. Indeed, reasoning as in the proof of Lemma \ref{lam}, there is $t_0'<t_0$ such that
$$G(t)\geq t^{p^{-}+\varepsilon}, \quad \text{for all }t\in (0, t_0'),$$and with $\varepsilon\in (0, p_B^{-}-p^{-})$. Hence, for $\rho \in (0, t_0')$ and $\|u\|_{2, G}=\rho$, 
$$\Phi_\lambda(u)\geq \|u\|_{2, G}^{p^{-}+\varepsilon}-C\lambda  \|u\|_{2, G}^{p_B^{-}}=\rho^{p_B^{-}}(\rho^{p^{-}+\varepsilon-p_B^{-}}-C\lambda).$$Thus, for any $\lambda>0$, and taking $\rho$ small enough, we have that for any $u\in \partial B_\rho(0)$. there holds $\Phi_\lambda(u)>0$. 

Also, taking  any $v\in C_0^{\infty}(\Omega)$ so that $0<\|v\|_{2, G}<1$, we get
$$\Phi_\lambda(v)\leq \max\left\lbrace \|v\|_{2, G}^{p^+}, \|v\|_{2, G}^{p^-}\right\rbrace-\lambda C \min\left\lbrace \|v\|_{2, G}^{p_B^+}, \|v\|_{2, G}^{p_B^-}\right\rbrace.$$As in the proof of Theorem \ref{eigenvalueinfty}, let
$$F(t):= \dfrac{\max\left\lbrace t^{p^+}, t^{p^-}\right\rbrace}{C \min\left\lbrace t^{p_B^+}, t^{p_B^-}\right\rbrace},$$then there holds
$$F(t)=\begin{cases} \dfrac{1}{C}t^{p^{-}-p_B^{+}}, \quad 0<t<1\\\dfrac{1}{C}t^{p^{+}-p_B^{-}}, \quad t>1.
\end{cases}$$Observe that by \eqref{hyp b 22}, $p^{-}-p_B^{+}<0$ and $p^{+}-p_B^{-}>0$, the range of $F$ is $(1/C, \infty)$. Hence, choosing 
\begin{equation}\label{cond lambda}
\lambda^{**}:=\dfrac{1}{C}>0,
\end{equation}we get that for any $\lambda\in (\lambda^{**},\infty)$, there is $v$ so that $\Phi_\lambda(v)<0$. The rest of the proof follows as for Theorem \ref{eigenvalue} applying the Ekeland's Variational Principle.  
\end{proof}

\begin{remark}
Under the assumption \eqref{hyp b 22}, we have

$$\dfrac{B(t)}{G(t)}\leq \dfrac{t^{p_B^{-}}}{t^{p^{-}+\varepsilon}}\to 0 \quad \text{as }t\to 0.$$Thus, in Theorem \ref{eigenvalueinfty2}, we have $B\ll G$ near $0$.
\end{remark}
\normalcolor


\begin{theorem}\label{eigenvalue3}
Let $G$ and $B$ be two $N$-functions satisfying \eqref{G1}. Moreover assume that $B \ll G^*$ and that
\begin{equation}\label{hyp b 2}
1<p^{-}<p^+<p_B^{-}<p_B^{+}.
\end{equation} 
Then, any $\lambda\in (0, \infty)$ is an eigenvalue of the problem \eqref{problem eigenvalue}.
\end{theorem}

\begin{proof}
Again, we will check that Lemma \ref{lemaa} and Lemma \ref{lemb} hold. Let $\lambda\in (0, \infty)$, take $\rho \in (0, 1)$, and let $\|u\|_{2, G}=\rho$. Following the calculations from \eqref{large calc}, we get
\begin{equation*}
\Phi_\lambda(u)\geq \rho^{p^+}-\lambda C \rho^{p_B^-}= \rho^{p_B^-}(\rho^{p^+-p_B^-}-\lambda C). 
\end{equation*} So, by \eqref{hyp b 2}, for $\rho$ small enough, there is $\alpha>0$ such that $\Phi_\lambda(u)\geq \alpha$ for any $\|u\|_{2, G}=\rho$.

Next, for $t>1$, take $v\in C_0^{\infty}(\Omega)$ such that $1<\|v\|_{2, G}$. Then,
\begin{equation}
\Phi_\lambda(tv)\leq t^{p^+}\|v\|_{2, G}^{p^+}-C\lambda t^{p_B^-} \|v\|_{2, G}^{p_B^{-}} \to -\infty, \quad \text{as }t\to \infty.
\end{equation}Hence, we conclude following the lines of the proof of Theorem \ref{eigenvalue}.
\end{proof}

\begin{remark}
Under the assumption \eqref{hyp b 2}, we have
$$\dfrac{B(t)}{G(t)}\leq \dfrac{t^{p_B^{-}}}{t^{p^{+}}}\to 0 \quad \text{as }t\to 0,$$and
$$\dfrac{G(t)}{B(t)}\leq \dfrac{t^{p^{+}}}{t^{p_B^{-}}}\to 0  \quad \text{as }t\to \infty.$$
Thus, in Theorem \ref{eigenvalue3}, we have $B\ll G$ near $0$ and $G \ll B$ near $\infty$.
\end{remark}

\section*{Acknowledgments}
This work was partially supported by CONICET PIP 11220210100238CO and
ANPCyT PICT 2019-03837. P. Ochoa and A. Silva are members of CONICET.

\end{document}